\documentclass[12pt]{amsart}
\usepackage{amssymb, amscd}
\usepackage{hyperref}

\newtheorem{Lemma}             {Lemma}
\newtheorem{Corollary}  [Lemma]{Corollary}
\newtheorem{Proposition}[Lemma]{Proposition}

\newtheorem{Theorem}    [Lemma]{Theorem}

\newcommand{\ov}{\overline}
\newcommand{\op}{\operatorname}

\title[Self-dual Brauer]{Self-dual modules in characteristic two and normal subgroups}

\author{Rod Gow}
\address{School of Mathematical Sciences,\\University College Dublin\\IRELAND}
\email{Rod.Gow@ucd.ie}

\author{John Murray}
\address{Department of Mathematics and Statistics\\Maynooth University\\IRELAND}
\email{John.Murray@mu.ie}

\date{August 19, 2020}

\begin{document}
\maketitle

\begin{abstract}
We prove Clifford theoretic results which only hold in characteristic $2$.

Let $G$ be a finite group, let $N$ be a normal subgroup of $G$ and let $\varphi$ be an irreducible $2$-Brauer character of $N$. We show that $\varphi$ occurs with odd multiplicity in the restriction of some self-dual irreducible Brauer character $\theta$ of $G$ if and only if $\varphi$ is $G$-conjugate to its dual. Moreover, if $\varphi$ is self-dual then $\theta$ is unique and the multiplicity is $1$.

Next suppose that $\theta$ is a self-dual irreducible $2$-Brauer character of $G$ which is not of quadratic type. We prove that the restriction of $\theta$ to $N$ is a sum of distinct self-dual irreducible Brauer character of $N$, none of which have quadratic type. Moreover, $G$ has no self-dual irreducible $2$-Brauer character of non-quadratic type if and only if $N$ and $G/N$ satisfy the same property.

Finally, suppose that $b$ is a real $2$-block of $N$. We show that there is a unique real $2$-block of $G$ covering $b$ which is weakly regular with respect to $N$.
\end{abstract}

\section{Statement of results}


\noindent Throughout the paper $G$ is a finite group and $N$ is a normal subgroup of $G$. We fix a $2$-modular system $(K,R,F)$ for $G$. So $R$ is a complete discrete valuation ring which has field of fractions $K$ of characteristic $0$ and residue field $F=R/J(R)$ of characteristic $2$. We will assume that $K$ and $F$ are splitting fields for all subgroups of $G$. For example, this holds if $K$ contains a primitive $|G|$-th root of unity. We use $r^*$ to denote the image of $r\in R$ in $F$. Each integer $m$ can be factored as $m=m_2m_{2'}$, where $m_2$ is a power of $2$ and $m_{2'}$ is odd. 

We use $\op{Irr}(G)$ to denote the irreducible $K$-characters of $G$. These have values in a cyclotomic subfield of $K$ which can be identified with a subfield of ${\mathbb C}$. So $\op{Irr}(G)$ can be identified with the irreducible complex characters of $G$. Next recall that the Brauer character of an $FG$-module is an $R$-valued class function defined on the $2$-regular (odd order) elements of $G$. The Brauer characters of the irreducible $FG$-modules are called the irreducible 2-Brauer characters of $G$. We use $\op{IBr}(G)$ to denote all such characters. The dual of a character $\theta$ is the character $\ov\theta$ defined by $\ov\theta(g):=\theta(g^{-1})$, for all $g\in G$. We say that $\theta$ is self-dual if $\theta=\ov\theta$. This holds if and only if $\theta$ is the character of some self-dual module.

Let $\theta$ be an irreducible Brauer character of $G$ and let $\varphi$ be an irreducible Brauer character of a subgroup $H$ of $G$. We say that $\theta$ lies over $\varphi$ if $\theta$ is constituent of the induced character $\varphi{\uparrow^G}$. Likewise we say that $\varphi$ lies under $\theta$ if $\varphi$ is a constituent of the restricted character $\theta{\downarrow_H}$. There is no analogue of Frobenius reciprocity for Brauer characters. So the fact that $\theta$ lies over $\varphi$ does not imply that $\varphi$ lies under $\theta$, and conversely. However these implications do hold if $H$ is normal in $G$ \cite[p155 and (8.7)]{N}. Our first result is:

\begin{Theorem}\label{T:conjugate_dual}
Let $\varphi$ be an irreducible $2$-Brauer character of\/ $N$. Then $\varphi$ lies under some self-dual\/ irreducible $2$-Brauer character $\theta$ of\/ $G$ if only if\/ $\varphi$ is $G$-conjugate to $\ov\varphi$. If such a self-dual $\theta$ exists, it can be chosen so that $\varphi$ occurs with odd multiplicity in $\theta{\downarrow_N}$.
\end{Theorem}

Our second result is:

\begin{Theorem}\label{T:real_extends}
Let $\varphi$ be a self-dual irreducible $2$-Brauer character of\/ $N$. Then
\begin{itemize}
 \item[(i)] $\varphi$ extends to its stabilizer in $G$, and exactly one of these extensions is self-dual.
 \item[(ii)] $G$ has a unique self-dual irreducible $2$-Brauer character $\theta$ such that $\varphi$ occurs with odd multiplicity in $\theta{\downarrow_N}$.
 \item[(iii)] $\varphi$ occurs with multiplicity $1$ in $\theta{\downarrow_N}$.
\end{itemize}
\end{Theorem}

We will refer to $\theta$ as the canonical irreducible Brauer character of $G$ lying over $\varphi$. The module form of this theorem is stated and proved in Theorem~\ref{T:self_dual_extends}~below.

Replacing the normal subgroup $N$ by a subnormal subgroup $H$, part (ii) of Theorem \ref{T:real_extends} still holds. However part (iii) need not hold. Indeed the multiplicity can be arbitrarily large, as J. M\"uller has pointed out. We outline his example at the end of Section \ref{S:Clifford}.

Theorem~\ref{T:real_extends}~is similar in flavour to a result of I. M. Richards. In \cite{R} he proved that when $G/N$ has odd order, each self-dual irreducible $K$-character of $N$ extends to its stabilizer in $G$, and has a unique self-dual extension.

{\bf Example:} Let $G$ be the semi-dihedral group of order $2^n\geq16$, and let $N$ be the centre of $G$. Then $N$ has a unique, hence self-dual, non-trivial irreducible $K$-character. However all faithful irreducible $K$-characters of $G$ have degree $2$ and none of them is self-dual. So neither Theorems~\ref{T:conjugate_dual}~nor~\ref{T:real_extends} generalize to irreducible $K$-characters nor to irreducible $p$-Brauer characters, for primes $p\ne2$.

Next recall that a non-trivial irreducible $2$-Brauer character of $G$ is said to have quadratic type if the corresponding $FG$-module affords a non-zero $G$-invariant quadratic form. Our first application is:

\begin{Theorem}\label{T:non_quadratic}
Let $\theta$ be a non-quadratic type self-dual irreducible $2$-Brauer character of\/ $G$ which does not lie over the trivial character of\/ $N$. Then $\theta{\downarrow_N}$ is a sum of non-quadratic type self-dual irreducible Brauer characters of\/ $N$, each occurring with multiplicity $1$.
\end{Theorem}

The second application is to blocks. For undefined notation, see Section \ref{S:blockCovering} below and for a full exposition of block theory, see Chapter 5 of \cite{NT}.

Let $B$ be a $2$-block of irreducible $K$-characters of $G$. Then the duals of the characters in $B$ form another $2$-block $B^o$, called the contragredient of $B$. We say that $B$ is real if $B=B^o$. Recall that $B$ is said to cover a $2$-block $b$ of $N$ if the restriction of an irreducible character in $B$ contains an irreducible character in $b$, and that $B$ is said to be weakly regular (with respect to $N$) if it has maximal defect among the blocks of $G$ which cover $b$.

\begin{Theorem}\label{T:block_covering}
Let $b$ be a $2$-block of $N$. Then
\begin{itemize}
 \item[(i)] $G$ has a real weakly regular $2$-block covering $b$ if and only if\/ $b$ is $G$-conjugate to $b^o$.
 \item[(ii)] $G$ has a unique real weakly regular $2$-block covering $b$ if\/ $b=b^o$.
\end{itemize}
\end{Theorem}

We prove (i) in Lemma~\ref{L:block_covering_part1} and (ii) in Lemma~\ref{L:block_covering_part2}.

Recall that corresponding to each irreducible $2$-Brauer character $\theta$, $G$ has a principal indecomposable character $\Phi_\theta$. Then $\Phi_\theta$ is a $K$-character of $G$ which vanishes off the $2$-regular elements of $G$. We use the following result, which is implicit in \cite[1.4]{GW93}, to prove Theorem \ref{T:block_covering}. As it may be of independent interest, we include a short proof here:

\begin{Lemma}\label{L:odd_height0}
Let $B$ be a $2$-block of\/ $G$. Then $B$ has an odd number of height $0$ irreducible Brauer characters $\theta$ such that $\Phi_\theta(1)_2=|G|_2$.
\begin{proof}
Let $D$ be a defect group of $B$. Then Brauer showed that $\frac{\op{dim}(B)}{|G||G:D|}$ is a unit in $R$. See \cite[5.10.1]{NT}. Now $\op{dim}(B)=\sum_{\theta\in\op{IBr}(b)}\Phi_\theta(1)\theta(1)$. It is known that $\Phi_\theta(1)/|G|$ and $\theta(1)/|G:D|$ belong to $R$, for all $\theta\in\op{IBr}(B)$. So Brauer's result gives us an identity in $F$:
$$
\sum_{\theta\in\op{IBr}(B)}\left(\frac{\Phi_\theta(1)}{|G|}\right)^*\left(\frac{\theta(1)}{|G:D|}\right)^*=\left(\frac{\op{dim}(B)}{|G||G:D|}\right)^*=1_F.
$$
The contribution of $\theta$ to the left hand side is $1_F$, if $\theta$ has height $0$ and $\Phi_\theta(1)_2=|G|_2$. Otherwise the contribution is $0_F$. So the lemma follows directly from the above equality.
\end{proof}
\end{Lemma}

Many 2-blocks have an odd number of height $0$ irreducible Brauer characters. For example, the main result of \cite{KOW} is that each 2-block of a symmetric group has a unique height $0$ irreducible Brauer character. Furthermore, it is known that each principal indecomposable character $\Phi$ of a finite solvable group satisfies $\Phi(1)_2=|G|_2$. So Lemma \ref{L:odd_height0} implies that each $2$-block of a finite solvable group has an odd number of height $0$ irreducible Brauer characters. However, as B. Sambale has pointed out, the faithful $2$-block of $3.{\rm Suz}.2$ has four height $0$ irreducible Brauer characters. Three of these satisfy the condition of Lemma~\ref{L:odd_height0} on their principal indecomposable character degree. T. Wada and J. M\"uller have independently noted this example, and the fact that it is unique in the \cite{MA} database.

\section{Real orbits of irreducible Brauer characters}

\noindent Recall that $g\in G$ is said to be real in $G$ if $xgx^{-1}=g^{-1}$, for some $x\in G$.  Similarly a conjugacy class of $G$ is real if its elements are real, and $2$-regular if its elements have odd order. Now $G$ acts on the conjugacy classes, the irreducible $K$-characters and the Brauer characters of its normal subgroup $N$. We say that a $G$-orbit of conjugacy classes of $N$ is real if its union is a real conjugacy class of $G$. Likewise we say that a $G$-orbit of irreducible Brauer characters of $N$ is real if it contains the duals of its characters.

\begin{proof}[Proof of Theorem \ref{T:conjugate_dual}]
It is clear that each self-dual irreducible Brauer character of $G$ lies over a real $G$-orbit of irreducible Brauer characters of $N$.

Suppose that $G$ has $\ell$ conjugacy classes of $2$-regular elements, with representatives $g_1,\dots,g_\ell$. Let $\theta_1,\dots,\theta_\ell$ be the irreducible Brauer characters of $G$ and let $\Phi_1,\dots\Phi_\ell$ be the corresponding principal indecomposable characters of $G$. The second orthogonality relations give
$$
\sum_{\chi\in\op{Irr}(G)}\chi(g_i^{-1})\chi(g_j)=\delta_{i,j}|\op{C}_G(g_i)|,\quad\mbox{for all $i,j\in\{1,\dots,\ell\}$.}
$$
Now for all $\chi\in\op{Irr}(G)$, we have $\chi(g_j)=\sum_{u=1}^\ell d_{\chi,\theta_u}\theta_u(g_j)$, where the $d_{\chi,\theta_u}$ are non-negative integers, called decomposition numbers. Then for all $u=1,\dots,\ell$, we have $\Phi_u(g_i)=\sum_{\chi\in\op{Irr}(G)}d_{\chi,\theta_u}\chi(g_i)$. It is known that $\frac{\Phi_u(g_i^{-1})}{|\op{C}_G(g_i)|}\in R$, for all $u,i$. So the above displayed equation can be rewritten in $R$ as 
\begin{equation}\label{E:Phi_theta}
\sum_{u=1}^\ell\frac{\Phi_u(g_i^{-1})}{|\op{C}_G(g_i)|}\theta_u(g_j)=\delta_{i,j},\quad\mbox{for all $i,j\in\{1,\dots,\ell\}$.}
\end{equation}
In particular the Brauer character table $[\theta_i(g_j)]$ of $G$ is a non-singular $\ell\times\ell$ matrix.\\ (We note that the proof of \cite[(2.18)]{N} shows that $\op{det}[\theta_i(g_j)]^2 = \pm\prod_{j=1}^\ell|\op{C}_G(g_j)|_{2'}$.)

Suppose that $G$ has $r$ real conjugacy classes of $2$-regular elements, which we may assume have representatives $g_1,\dots,g_r$. We choose notation so that $\theta_1,\dots,\theta_r$ are the self-dual irreducible Brauer characters of $G$. Then the self-dual Brauer character table of $G$ is the $r\times r$ submatrix $T:=[\theta_i(g_j)]$ of the Brauer character table. Suppose that $u\in\{r+1,\dots,\ell\}$. Then there is a unique $\ov{u}\in\{r+1,\dots,\ell\}$, with $\ov u\ne u$, such that $\theta_{\ov{u}}=\ov\theta_u$. So $\frac{\Phi_u(g_i^{-1})}{|\op{C}_G(g_i)|} \theta_u(g_j) = \frac{\Phi_{\ov{u}} (g_i^{-1})}{|\op{C}_G(g_i)|} \theta_{\ov{u}}(g_j)$, for all $i,j$. So the contribution of the summands indexed by $u$ and $\ov{u}$ to \eqref{E:Phi_theta} is $0$ mod $J(R)$, and we deduce that
$$
\sum_{u=1}^r\frac{\Phi_u(g_i^{-1})}{|\op{C}_G(g_i)|}\theta_u(g_j)\equiv\delta_{i,j}\quad\mbox{mod $J(R)$},\quad\mbox{for all $i,j\in\{1,\dots,r\}$.}
$$
As $R$ is a local ring, it follows that $T$ is invertible, with inverse congruent to the $r\times r$-matrix $\left[\frac{\Phi_i(g_j^{-1})}{|\op{C}_G(g_j)|}\right]$ mod $J(R)$. In particular $\op{det}(T)\not\in J(R)$.

Now suppose that $G$ has $t$ real conjugacy classes of $2$-regular elements which are contained in $N$, with representatives $n_1,\dots,n_t$. We relabel the $\theta_1,\dots,\theta_r$ so that the $t\times t$ submatrix $S:=[\theta_i(n_j)]$ of $T$ satisfies $\op{det}(S)\not\in J(R)$.

For $i=1,\dots,t$, let $\varphi_i$ be an irreducible Brauer character of $N$ which is a constituent of $\theta_i{\downarrow_N}$, and set $\hat\varphi_i$ as the sum of the distinct $G$-conjugates of $\varphi_i$. Then $\theta_i{\downarrow_N}=e_i\hat\varphi_i$, for some positive integer $e_i$. The non-singularity of $S$ implies that all the multiplicities $e_1,\dots,e_t$ are odd and $\varphi_1,\dots,\varphi_t$ lie in distinct $G$-orbits. Moreover each of these orbits is real, as each $\theta_i$ is self-dual.

By the non-singularity of the Brauer character table of $N$ and Brauer's permutation lemma, $G$ has $t$ real orbits on the irreducible Brauer characters of $N$. So $\varphi_1,\dots,\varphi_t$ represent all real $G$-orbits of irreducible Brauer characters of $N$.

Our work above shows that if $\varphi$ is an irreducible Brauer character of $N$ which is $G$-conjugate to $\ov\varphi$, then $G$ has a self-dual irreducible Brauer character $\theta$ such that $\varphi$ occurs with odd multiplicity in $\theta{\downarrow_N}$. This concludes the proof.
\end{proof}

\section{Clifford Theory for self-dual irreducible modules}\label{S:Clifford}

\noindent We prove Theorem \ref{T:real_extends} in this section. Recall that the dual of an $FG$-module $V$ is the $FG$-module $V^*:=\op{Hom}_F(V,F)$. We say that $V$ is self-dual if $V\cong V^*$ as $FG$-modules. For the reader's convenience, we begin by stating a module version of Clifford's Theorem:

\begin{Lemma}[Clifford 1937]\label{L:Clifford}
Let $F$ be an arbitary field, let $V$ be an irreducible $FG$-module and let $W$ be an irreducible submodule of\/ $V{\downarrow_N}$. Set $T$ as the stabilizer of\/ $W$ in $G$. Then
\begin{enumerate}
\item[(i)] $V{\downarrow_N}\cong e(W_1\oplus\dots\oplus W_n)$ for some integer $e>0$, where $W_1,\dots,W_n$ are the distinct $G$-conjugates of\/ $W$. In particular $V{\downarrow_N}$ is semi-simple.
\item[(ii)] Let $U$ be the sum of all submodules of\/ $V{\downarrow_N}$ which are isomorphic to $W$. Then $U$ is an irreducible $FT$-module, $U{\downarrow_N}=eW$ and $U{\uparrow^G}=V$.
\item[(iii)] If\/ $W$ is absolutely irreducible, it extends to a projective $FT$-module $\hat W$ and $U\cong P\otimes\hat W$, for some projective $F(T/N)$-module $P$.
\item[(iv)] If\/ $W$ extends to an $FT$-module $\hat W$, then the distinct irreducible $FG$-modules lying over $W$ are $(S_1\otimes \hat W) {\uparrow^G}, \dots, (S_t\otimes \hat W){\uparrow^G}$ where $S_1,\dots,S_t$ are the distinct irreducible $F(T/N)$-modules.
\end{enumerate} 

\begin{proof}
See for example Huppert and Blackburn, Finite Groups II, VII, 9.12.
\end{proof}	
\end{Lemma}

Next we point out that the following result, generally known as Fong's Lemma, holds for an arbitrary field $F$ of characteristic $2$. In particular $F$ need not be perfect:

\begin{Lemma}\label{L:Fong}
Let $G$ be a finite group, let $F$ be an arbitrary field of characteristic $2$ and let $V$ be a non-trivial self-dual irreducible $FG$-module. Then $V$ affords a non-degenerate $G$-invariant alternating bilinear form. In particular $\op{dim}(V)$ is even.
\begin{proof}
The self-duality of $V$ is equivalent to the existence of a non-degenerate $G$-invariant bilinear form $B:V\times V\rightarrow F$. If $B$ is not symmetric, set $\hat B(v_1,v_2)=B(v_1,v_2)+B(v_2,v_1)$, for all $v_1,v_2\in V$. Then $\hat B$ is a $G$-invariant non-zero symmetric bilinear form on $V$. Now $\op{rad}(\hat B)$ is a submodule of $V$ and $V$ is irreducible. So $\op{rad}(\hat B)=0$ and $\hat B$ is non-degenerate.

The first paragraph shows that $V$ affords a $G$-invariant non-degenerate symmetric bilinear form, henceforth denoted $B$. We claim that $B$ is alternating, meaning $B(v,v)=0$, for all $v\in V$. For suppose otherwise. Set $Q(v):=B(v,v)$, for all $v\in V$. Then $Q$ is a non-zero $G$-invariant quadratic form on $V$. Now $Q$ is additive, as for all $v_1,v_2\in V$, we have
$$
\begin{aligned}
Q(v_1+v_2)	&=B(v_1+v_2,v_1+v_2)\\
			&=B(v_1,v_1)+B(v_1,v_2)+B(v_2,v_1)+B(v_2,v_2)\\
			&=Q(v_1)+Q(v_2),\quad\mbox{using $B(v_1,v_2)+B(v_2,v_1)=0$.}
\end{aligned}
$$
Moreover $Q(\lambda v)=\lambda^2 Q(v)$, for all $\lambda\in F$ and $v\in V$. Define $U:=\{v\in V\mid Q(v)=0\}$. Then our work shows that $U$ is a submodule of $V$. But $U\ne V$, as $Q\ne0$. So $U=0$, by irreducibility of $V$. Let $v\in V$ and $g\in G$. Then $Q(gv+v)=Q(gv)+Q(v)=0$, as $Q$ is additive and $G$-invariant. So $gv+v\in U$, whence $gv=v$. But then $G$ acts trivially on $V$, contrary to hypothesis. This proves our claim.

The final statement follows as every symplectic vector space has even dimension.
\end{proof}
\end{Lemma}

\begin{Corollary} \label{L:Radical}
Let $G$ be a finite group and let $F$ be an arbitrary field of characteristic $2$. Then the radical of $FG$ has odd codimension in $FG$.
\end{Corollary}

\begin{proof}
We use $\op{rad}(FG)$ to denote the radical of $FG$, which is the intersection of the annihilators of all irreducible $FG$-modules. 
Suppose first that $F$ is a splitting field for $G$. Let $\theta_1,\ldots,\theta_\ell$ be the irreducible Brauer characters of $G$, with $\theta_1,\ldots,\theta_r$ precisely the self-dual characters, and $\theta_1$ the trivial
Brauer character. We have 
\[
\dim (FG)-\dim(\op{rad}(FG))=\theta_1(1)^2+\cdots +\theta_\ell(1)^2.
\]
Now $\theta_1(1)=1$ and $\theta_i(1)$ is even for $2\leq i\leq r$, by Fong's Lemma. If $i>r$, we may pair $\theta_i$ with its dual, and these two characters have the same degree. It is now clear that
\[
\theta_1(1)^2+\cdots +\theta_\ell(1)^2
\]
is odd, and the result follows in this case.

Now suppose that $F$ is any field of characteristic $2$. Set $E=F(\omega)$, where $\omega$ is a primitive $|G|_{2'}$-th root of unity in an extension field of $F$. Then $E$ is a splitting field for $G$ and a finite separable extension of $F$. As $FG$ contains an $E$-basis of $EG$, it is a standard fact that $\op{rad}(EG)$ is the $E$-span of $\op{rad}(FG)$. In particular $\op{dim}_E(\op{rad}(EG))=\op{dim}_F(\op{rad}(FG))$. The first part shows that $\dim_E(EG)-\op{dim}_E(\op{rad}(EG))$ is odd. So $\dim_F (FG)-\op{dim}_F(\op{rad}(FG))$ is odd in this case also.
\end{proof} 

For the rest of this section $F$ is a perfect field of characteristic $2$. Here is the module version of Theorem \ref{T:real_extends}:

\begin{Theorem}\label{T:self_dual_extends}
Let $W$ be a self-dual irreducible $FN$-module. Then $W$ extends to its stabilizer in $G$, and there is a unique extension $\hat W$ which is self-dual.

Set $V:=\hat W{\uparrow^G}$. Then $V$ is a self-dual irreducible $FG$-module, and $V{\downarrow_N}\cong W_1\oplus\dots\oplus W_n$, where $W_1,\dots,W_n$ are the distinct $G$-conjugates of\/ $W$. Moreover $V$ is the unique self-dual irreducible $FG$-module such that $W$ occurs with odd multiplicity in $V{\downarrow_N}$.
\end{Theorem}

\begin{proof}
We may assume that $W$ is non-trivial and $G$-invariant. As $W$ is a self-dual $FN$-module, it affords a non-degenerate $N$-invariant bilinear form $B:W\times W\to F$. An application of Schur's Lemma shows that $B$ is unique up to scaling.

Let $X:N\rightarrow\op{GL}(W)$ be the $F$-representation given by $W$. For each $g\in G$, we define the conjugate representation $X^g$ of $N$ by
$$
X^g(n)=X(gng^{-1}),\quad\mbox{for all $n\in N$.}
$$
Then $X$ and $X^g$ are equivalent representations, as $W$ is $G$-invariant. So there is $Y(g)\in\op{GL}(W)$ such that 
$$
Y(g)X(n)=X^g(n)Y(g),\quad\mbox{for all $n\in N$.}
$$
We choose $Y(g)=X(g)$, whenever $g\in N$. Now for all $g,h\in G$, we have
$$
[Y(gh)^{-1}Y(g)Y(h)]\,X(n)=X(n)\,[Y(gh)^{-1}Y(g)Y(h)].
$$
So by Schur's Lemma there is a non-zero $\alpha(g,h)\in F$ such that
\begin{equation}\label{E:alpha}
Y(gh)=\alpha(g,h)Y(g)Y(h).
\end{equation}
Then $\alpha:G\times G\to F^\times$ is an $F$-valued cocycle and $Y$ is a projective representation of $G$ which extends $X$.

Next, for all $g\in G$, we define the bilinear form $B^g:W\times W\to F$ by
$$
B^g(u,v)=B(Y(g)u,Y(g)v),\quad\mbox{for all $u,v\in W$.}
$$
Then for all $n\in N$ we have
\begin{eqnarray*}
B^g(X(n)u,X(n)v)&=&B(Y(g)X(n)u,Y(g)X(n)v)\\
&=&B(X(gng^{-1})Y(g)u,X(gng^{-1})Y(g)v)\\
&=&B(Y(g)u,Y(g)v),\quad\mbox{as $B$ is $X$-invariant}\\
&=&B^g(u,v).
\end{eqnarray*}
This shows that $B^g$ is $X$-invariant. As $B$ is unique up to scalars
\begin{equation}\label{E:lambda}
B^g=\lambda(g)B,\quad\mbox{for some $\lambda(g)\in F^\times$.}
\end{equation}
As $B$ is $N$-invariant, we have $\lambda(n)=1$, for all $n\in N$. 

Now for all $g,h\in G$ we have $B^{gh}=\lambda(gh)B$. On the other hand
\begin{eqnarray*}
B^{gh}(u,v)
&=& B(Y(gh)u,Y(gh)v)\\
&=& B(\alpha(g,h)Y(g)Y(h)u,\alpha(g,h)Y(g)Y(h)v),\quad\mbox{by \eqref{E:alpha},}\\
&=& \alpha(g,h)^2\lambda(g)\lambda(h)B(u,v),\quad\mbox{by \eqref{E:lambda}.}
\end{eqnarray*}
Comparing these expressions, we see that 
\begin{equation}\label{E:alpha_lambda}
\lambda(gh)=\alpha(g,h)^2\lambda(g)\lambda(h).
\end{equation}

Since $F$ is perfect, for each $g$ in $G$, there exists $\mu(g)\in F$ such that $\mu(g)^2=\lambda(g)$. Set $\hat{Y}(g)=\mu(g)^{-1}Y(g)$ for all $g\in G$. Then $\hat{Y}$ is a projective representation of $G$ which extends $X$. Moreover $\hat Y$ corresponds to the cocycle $\beta(g,h):=\mu(g)\mu(h)\mu(gh)^{-1}\alpha(g,h)$.

Now $\beta(g,h)^2=\lambda(g)\lambda(h)\lambda(gh)^{-1}\alpha(g,h)^2=1$ and $\op{char}(F)=2$. So $\beta(g,h)=1$, for all $g,h\in G$. This means that $\hat{Y}$ is an $F$-representation of $G$ which extends $X$.

If we now consider the action of the elements $\hat{Y}(g)$ on the bilinear form $B$, a repetition of an earlier argument shows that for all $u$ and $v$ in $W$, and all $g\in G$,
$$
B(\hat{Y}(g)u,\hat{Y}(g)v)=\epsilon(g)B(u,v)
$$
for some nonzero scalar $\epsilon(g)$. The fact that $\hat{Y}$ is a representation of $G$, and $B$ is $N$-invariant now implies that $\epsilon$ is a homomorphism $G/N\rightarrow F^\times$. 

Finally, as $F$ has characteristic 2, $\epsilon$ has odd order in the character group of $G/N$. So as $F$ is perfect, $\epsilon=\delta^2$ for a unique homomorphism $\delta:G/N\rightarrow F^\times$. Then if we set ${\hat X}(g)=\delta(g)^{-1}\hat{Y}(g)$, we find that ${\hat X}$ is also an $F$-representation of $G$ which extends $X$. Moreover $\hat X$ is self-dual, as we can easily check that it leaves $B$ invariant.

Let $\hat W$ be the irreducible self-dual $FG$-module corresponding to $\hat X$. So $\hat W$ extends $W$. Then $S\otimes{\hat W}$ give all irreducible $FG$-modules lying over $W$, as $S$ ranges over all irreducible $FG/N$-modules. Recall that $S\otimes{\hat W}\cong S'\otimes{\hat W}$ if and only if $S\cong S'$. So $S\otimes{\hat W}$ is self-dual if and only if $S$ is self-dual. Fong's Lemma implies that $\op{dim}(S\otimes{\hat W})$ is an even multiple of $\op{dim}(\hat W)$, if $S$ is non-trivial and self-dual. So $\hat W$ is the unique extension of $W$ to $G$ which is self-dual. The statements about $V$ are now consequences of Lemma \ref{L:Clifford}.
\end{proof}

We will refer to $V$ as the canonical self-dual irreducible $FG$-module lying over $W$. Our Corollary, which is probably known, is an analogue of Richards' Theorem \cite{R} for irreducible $2$-Brauer characters:

\begin{Corollary}\label{C:odd_quotient}
Suppose that $|G:N|$ is odd. Then
\begin{enumerate}
\item[(i)] If\/ $W$ is a self-dual irreducible $FN$-module, then $W{\uparrow^G}$ has a unique self-dual composition factor, up to isomorphism.
\item[(ii)] If\/ $V$ is a self-dual irreducible $FG$-module then $V{\downarrow_N}$ is a sum of distinct self-dual irreducible $FN$-modules.
\end{enumerate}
In particular induction-restriction defines a natural correspondence between the self-dual irreducible $FG$-modules and the $G$-orbits of self-dual irreducible $FN$-modules. 
\begin{proof}
We may assume that $W$ is $G$-invariant. Let $\hat W$ be the unique self-dual irreducible $FG$-module which extends $W$. Then all irreducible $FG$-modules lying over $W$ have the form $U\otimes W$, for some irreducible $FG/N$-module $U$. But $G/N$ has odd order. So $U$ is self-dual if and only if it is trivial. It follows that $\hat W$ is the unique self-dual irreducible $FG$-module lying over $W$. All composition factors of $W{\uparrow^G}$ lie over $W$. So (i) holds.

For (ii), write $V{\downarrow_N}=e(W_1\oplus\dots\oplus W_t)$, where $e,t\geq1$ and $W_1,\dots,W_t$ are distinct irreducible $FN$-modules. Now $V{\downarrow_N}$ is a self-dual irreducible $FN$-module. So for each $i=1,\dots,t$, $W_i^*\cong W_j$, for some $j=1,\dots,t$. The map $i\mapsto j$ is an involutary permutation of $\{1,\dots,t\}$. But $t$ is odd, as by Clifford Theory it divides $|G:N|$. So there exists $i$ such that $W_i^*\cong W_i$, whence all $W_1,\dots,W_t$ are self-dual. Now by part (i), $V$ is the unique self-dual irreducible $FG$-module lying over $W_1$. It then follows from Theorem \ref{T:self_dual_extends} that $e=1$ i.e. $V{\downarrow_N}=W_1\oplus\dots\oplus W_t$.

The last statement follows from (i) and (ii).
\end{proof}
\end{Corollary}

We note that Lemma \ref{L:Clifford} can be used to show that the restriction of an irreducible $FG$-module to a subnormal subgroup of $G$ is semi-simple. Part (ii) of Theorem \ref{T:real_extends} extends to subnormal subgroups:

\begin{Corollary}\label{C:subnormal}
Let $H$ be a subnormal subgroup of\/ $G$ and let $U$ be a self-dual irreducible $FH$-module. Then there is a unique self-dual irreducible $FG$-module $V$ such that $U$ occurs with odd multiplicity in $V{\downarrow_H}$.
\begin{proof}
We use induction on $|G:H|$. By Theorem \ref{T:self_dual_extends} we may assume that $H$ is not a normal subgroup of $G$. So there exists $H\subsetneq N\subsetneq G$ such that $H$ is subnormal in $N$ and $N$ is  normal in $G$. As $|N:H|<|G:H|$, our inductive assumption implies that there is a unique self-dual  irreducible $FN$-module $W$ such that $U$ has odd multiplicity in $W{\downarrow_H}$. Let $V$ be the canonical $FG$-module over $W$.

Now $V{\downarrow_N}=W\oplus W_2\oplus\dots\oplus W_t$ is the sum of the distinct $G$-conjugates of $W$. As $W$ is self-dual, all the $W_i$ are self-dual. By choice of $W$, $U$ appears with even multiplicity in $W_i{\downarrow_H}$, for $i=2,\dots,t$. Since $V{\downarrow_H}=W{\downarrow_H}\oplus W_2{\downarrow_H}\oplus\dots\oplus W_t{\downarrow_H}$, we deduce that $U$ appears with odd multiplicity in $V{\downarrow_H}$.

Now let $V'$ be a self-dual irreducible $FG$-module such that $U$ occurs with odd multiplicity in $V'{\downarrow_H}$. Write $V'{\downarrow_N}=e(W_1'\oplus\dots\oplus W_s')$, for some odd integer $e$, where $W_1',\dots,W_s'$ are distinct irreducible $FN$-modules. We claim that one and hence all $W_i'$ are self-dual. For suppose otherwise. As $V'{\downarrow_N}$ is self-dual, for each $i$ there is a unique $j\ne i$ such that $W_i^{'*}\cong W_j'$. Then $U$, being self-dual, occurs with the same multiplicity in $W_i'{\downarrow_H}$ as in $W_j'{\downarrow_H}$. So $U$ occurs with even multiplicity in $e(W_i'\oplus W_j')$, and hence with even multiplicity in $V'{\downarrow_H}$. This contradiction proved our claim. So $e=1$ and $V'$ is the canonical $FG$-module over each $W_i'$. Now we may assume that $U$ appears with odd multiplicity in $W_1'{\downarrow_H}$. So $W\cong W_1'$, by uniqueness of $W$ over $U$, and then $V'\cong V$, by uniqueness of $V$ over $W$.
\end{proof}
\end{Corollary}

We thank J. M\"uller for allowing us to include the following example. In the context of the Corollary it shows that the multiplicity of $U$ in $V{\downarrow_H}$ can be arbitrarily large.

{\bf Example:} Let $p$ be an odd prime, let $F=\op{GF}(p)$ and let $E=\op{GF}(p^2)$. Considering the additive group of $E$, let $Z$ be the group of scalars in $\op{GL}(E)\cong\op{GL}_2(F)$ and let $S$ be a Sylow $2$-subgroup of $\op{GL}(E)$. We choose an involution $t\in S$, depending on the residue class of $p$ mod $4$, as follows.

If $p\equiv1$ (mod $4$), let $C$ be a Sylow $2$-subgroup of the multiplicative group $F^\times$. Then we may assume that
$S=\left\{\begin{bmatrix}a&0\\0&b\end{bmatrix},\begin{bmatrix}0&a\\b&0\end{bmatrix}\mid a,b\in C\right\}$ and we take $t=\begin{bmatrix}0&1\\1&0\end{bmatrix}$.

If $p\equiv3$ (mod $4$), we take $t$ to be the Frobenius map $\lambda\mapsto\lambda^p$, for all $\lambda\in E$. So $S=C\rtimes\langle t\rangle$, where $C$ is a Sylow $2$-subgroup of the multiplicative group $E^\times$.

Now let $e_1,e_2\in E$ be non-trivial, such that $t(e_1)=e_1$ and $t(e_2)=-e_2$. So $E=Fe_1\oplus Fe_2$, as $F$-vector spaces. For our example we take $G=E\rtimes ZS$ and its subgroup $H=F e_2\rtimes\langle t\rangle$. Then $H$ is a dihedral group which is subnormal in $G$, for example because it is normal in $E\rtimes(Z\times\langle t\rangle)$, and $G/(E\rtimes Z)$ is nilpotent (as it is a $2$-group).

As $E$ is abelian, each of its irreducible $K$-characters is linear and as $|E|$ is odd, each $K$-character may be regarded a $2$-Brauer character. Let $\theta\in\op{Irr}(E)$ have kernel $Fe_2$. Then $\theta$ has stabilizer $E\rtimes\langle t\rangle$ in $G$, from our descriptions of $Z$ and $S$. Let $\hat\theta$ be the unique extension of $\theta$ to a Brauer character of $E\rtimes\langle t\rangle$. Then $\psi:=\hat\theta{\uparrow^G}$ is an irreducible Brauer character of $G$ and $\psi{\downarrow_E}$ is the sum of $\psi(1)$ distinct $G$-conjugates of $\theta$.

Let $g\in ZS$ and let $\mu$ be a nontrivial linear character of $Fe_2$. So $\phi:=\mu{\uparrow^H}$ is an irreducible $2$-Brauer character of $H$. Now $ZS$ acts on the $1$-dimensional $F$-subspaces of $E$, and $Z\langle t\rangle$ is the stabilizer of the subspace $Fe_1$ in $ZS$. We set $s=|ZS:Z\times\langle t\rangle|$.

Now $\theta^g$ is trivial on $Fe_2$ if and only if $g\in Z\langle t\rangle$. Consequently $\psi{\downarrow_H}$ contains the trivial Brauer character of $H$ with multiplicity $|Z|=p-1$. Suppose that $g\not\in Z\langle t\rangle$. Then $\theta^g{\downarrow_{Fe_2}}$ is a non-trivial character of $Fe_2$. So there is a unique $z\in Z$ such that $\theta^{gz}{\downarrow_{Fe_2}}=\mu$. It follows from this that $\psi{\downarrow_H}$ contains $\phi$ with multiplicity $s-1$. Finally, $p^2-1$ has $2$-part $2s\geq8$. So $\psi$ is the canonical character of $G$ lying over $\phi$, with odd multiplicity $s-1\geq3$.



\section{Irreducible self-dual modules of non-quadratic type}

\noindent Let $V$ be a non-trivial self-dual irreducible $FG$-module. Then by Lemma \ref{L:Fong}, $V$ affords a non-degenerate $G$-invariant alternating form $B$. Let $Q:V\rightarrow F$ be a quadratic form which polarizes to $B$. This means that $Q(\lambda v_1)=\lambda^2 Q(v_1)$ and $Q(v_1+v_2)=Q(v_1)+B(v_1,v_2)+Q(v_2)$, for all $\lambda\in F$ and $v_1,v_2\in V$. However, contrary to what happens when $\op{char}(F)\ne2$, $Q$ is not uniquely determined by $B$. In particular $Q$ need not be $G$-invariant. 

On the other hand, for many $FG$-modules each $G$-invariant quadratic form is uniquely determined by its polarization:

\begin{Lemma}\label{L:unique_quadratic}
Let $G$ be a finite group and let $F$ be an arbitrary field of characteristic $2$. Suppose that $V$ is an $FG$-module which affords a non-degenerate $G$-invariant alternating bilinear form $B$ but $V$ has no trivial quotient. Then $V$ affords at most one $G$-invariant quadratic form which polarizes to $B$.
\begin{proof}
Let $Q_1$ and $Q_2$ be $G$-invariant quadratic forms on $V$ which polarize to $B$. Then $P:=Q_1+Q_2$ is a $G$-invariant quadratic form which polarizes to $2B=0$. Thus $P(\lambda v_1)=\lambda^2 P(v_1)$, and $P(v_1+v_2)=P(v_1)+P(v_2)$, for all $\lambda\in F$ and $v_1,v_2\in V$.

Set $U:=\{m\in V\mid P(v)=0\}$. Then $U$ is a submodule of $V$. Let $g\in G$ and $v\in V$. Then $gv+v\in U$, as  $P(gv+v)=P(gv)+P(v)=0$. So $G$ acts trivially on $V/U$, whence $U=V$ by our hypothesis on $V$. We conclude that $P=0$, or equivalently $Q_1=Q_2$.
\end{proof}
\end{Lemma}

Recall the notion of a canonical irreducible $FG$-module introduced after Theorem \ref{T:self_dual_extends}.

\begin{Proposition}\label{P:canonical_quadratic_type}
Let $W$ be a non-trivial self-dual  irreducible $FN$-module and let $V$ be the canonical self-dual irreducible $FG$-module lying over $W$. Then $V$ has quadratic type if and only if\/ $W$ has quadratic type.
\end{Proposition}

\begin{proof}
We adopt the notation of Theorem \ref{T:self_dual_extends}. So $W$ has a unique self-dual extension $\hat W$ to its stabilizer $T$ and $V={\hat W}{\uparrow^G}$. Also $V{\downarrow_N}$ is the sum of the distinct $G$ conjugates of $W$, each occurring with multiplicity $1$. So we can identify $W$ with an $F$-subspace of $V$.

Suppose first that $V$ affords a $G$-invariant quadratic form $Q$, and let $B$ be its polarization. So $B$ is a non-degenerate $G$-invariant alternating form on $V$. Let $W^\perp=\{v\in V\mid B(v,w)=0,\mbox{ for all $w\in W$}\}$. Then $W^\perp$ is a submodule of $V{\downarrow_N}$ and $V{\downarrow_N}/W^\perp\cong W^*\cong W$ as $FN$-modules. As $W$ occurs with multiplicity $1$ in $V{\downarrow_N}$, we deduce that $W\cap W^\perp=0$. So the restriction $B{\downarrow_W}$ to $W$ is non-degenerate. Moreover the restriction $Q{\downarrow_W}$ of $Q$ to $W$ is an $N$-invariant quadratic form on $W$ which polarizes to $B{\downarrow_W}$. So $W$ has quadratic type.

Conversely, suppose that $W$ affords a non-degenerate $N$-invariant quadratic form $q$, and let $b$ be its polarization. Then $b$ is a non-degenerate $N$-invariant alternating form on $W$. We identify $W$ and $\hat{W}$ as $F$-vector spaces. As $\hat W$ is self-dual and irreducible, it affords a $T$-invariant non-zero bilinear form, say $b'$. Now all $N$-invariant non-zero bilinear forms on $W$ are scalar multiples of each other, as $W$ is irreducible. So $b$ is a scalar multiple of $b'$, and in particular $b$ is $T$-invariant.
 
For $t\in T$, we define a quadratic form $q^t$ on $W$ by setting
$$
q^t(w):=q(tw),\quad\mbox{for all $w\in W$.}
$$
It is clear that $q^t$ is $N$-invariant, and also that $q^t$ polarizes to $b$, as $b$ is $T$-invariant. So $q^t=q$, according to Lemma \ref{L:unique_quadratic}. This establishes that $q$ is $T$-invariant, and shows that $\hat W$ is of quadratic type.

Next, we may decompose $V={\hat W}{\uparrow^G}$ as $F$-vector space
$$
V=(g_1\otimes\hat W)\oplus(g_1\otimes\hat W)\oplus\dots\oplus(g_n\otimes\hat W),
$$
where $g_1,\dots,g_n$ is a transversal to $T$ in $G$. By a standard procedure, we may define the induced forms $b{\uparrow^G}$ and $q{\uparrow^G}$ on $V$ using
$$
\begin{aligned}
b{\uparrow^G}\left(\sum_{i=1}^n g_i\otimes w_i,\sum_{j=1}^n g_j\otimes x_j\right)
&=\sum_{i=1}^n b(w_i,x_i),\\
q{\uparrow^G}\left(\sum_{i=1}^n g_i\otimes w_i\right)
&=\sum_{i=1}^n q(w_i),
\end{aligned}
$$
for all $w_i,x_j\in\hat W$. Then $b{\uparrow^G}$ is a $G$-invariant alternating bilinear form on $V$ and $q{\uparrow^G}$ is a $G$-invariant quadratic form on $V$ which polarizes to $b{\uparrow^G}$. So $V$ is of quadratic type.
\end{proof}

\begin{Proposition} \label{P:even_quadratic_type}
Let $V$ be a self-dual irreducible $FG$-module and suppose that some self-dual irreducible $FN$-module $W$ occurs with multiplicity $e>1$ in $V{\downarrow_N}$. Then $e$ is even and $V$ has quadratic type.
\end{Proposition}

\begin{proof}
We adopt the notation of Theorem \ref{T:self_dual_extends}. So $W$ has a unique self-dual extension $\hat W$ to its stabilizer $T$ and $V=(S\otimes{\hat W}){\uparrow^G}$, where $S$ is a self-dual irreducible $F(T/N)$-module. Now $W$ occurs with multiplicity $1$ in $({\hat W}{\uparrow^G}){\downarrow_N}$. So the multiplicity $e$ of $W$ in $V{\downarrow_N}$ equals $\op{dim}(S)$. As $e>1$, we deduce that $S$ is non-trivial. But then $\op{dim}(S)$ is even, according to Lemma \ref{L:Fong}. Finally, since $S$ and $\hat{W}$ are both non-trivial and self-dual, $S\otimes\hat{W}$ has quadratic type, by the remark below. This in turn implies that $V=(S\otimes{\hat W}){\uparrow^G}$ has quadratic type.
\end{proof}

{\bf Remark:} Suppose that $U$ and $V$ are $FG$-modules which afford non-degenerate $G$-invariant alternating bilinear forms $B_U$ and $B_V$, respectively. According to Sin and Willems \cite[Proposition 3.4]{SW} there is a quadratic form $Q$ on $U\otimes_FV$, which polarizes to $B_U\otimes B_V$, such that $Q(u\otimes v)=0$, for all $u\in U$ and $v\in V$. These properties uniquely specify $Q$. For, if $u_1,\dots,u_n$ and $v_1,\dots,v_m$ are bases for $U$ and $V$, respectively, then for all $\lambda_{ij}\in F$
$$
Q\left(\sum\lambda_{ij}u_i\otimes v_j\right):=\sum\lambda_{ij}\lambda_{i'j'}B_U(u_i,u_{i'})B_V(v_j,v_{j'}),
$$
where $i,i'$ range over $1,\dots,n$ and $j,j'$ over $1,\dots,m$. Any basic tensor $u\otimes v$ can be written as $\sum\alpha_i\beta_ju_i\otimes v_j$, for scalars $\alpha_i,\beta_j$. Then in the expression for $Q(u\otimes v)$, the term indexed by $(i,j),(i',j')$ can be cancelled with the term indexed by $(i',j),(i,j')$, for $i\ne i'$. Likewise pairs of terms with $j\ne j'$ cancel. Finally, all terms with $i=i'$ or $j\ne j'$ are zero as $B_U$ and $B_V$ are alternating. It is now clear that $Q$ is $G$-invariant.

We turn our attention to those  irreducible $FN$-modules that are not self-dual but are $G$-conjugate to their duals. To investigate these, we require a familiar concept.

Suppose that $W$ is an irreducible $FN$-module, with stabilizer $T$ in $G$. Then
$$
T^*=\{ g\in G\mid W^g\cong W\mbox{ or }W^g\cong W^*\}
$$
is a subgroup of $G$ containing $T$, called the \emph{extended stabilizer} of $W$. If $W$ and $W^*$ are non-isomorphic and $G$-conjugate, then $|T^*:T|=2$. Otherwise $T=T^*$.

\begin{Proposition}\label{P:hyperbolic_quadratic}
Let $W$ be an  irreducible $FN$-module which is not self-dual. Then all self-dual irreducible $FG$-modules lying over $W$ are of quadratic type.
\end{Proposition}

\begin{proof}
If $W$ is not $G$-conjugate to $W^*$, there are no self-dual irreducible $FG$-modules lying over $W$. So we may assume that $W$ is $G$-conjugate to $W^*$ and that $|T^*:T|=2$.

Let $V$ be a self-dual irreducible $FG$-module lying over $W$. Then $V=U{\uparrow^G}$, where $U$ is the unique irreducible submodule of $V{\downarrow_T}$ lying over $W$. Likewise $V\cong V^*=(U^*){\uparrow^G}$. So $U^*$ is the unique irreducible submodule of $V{\downarrow_T}$ lying over $W^*$. Note that $U^*\not\cong U$, as $W$ and $W^*$ are not $T$-conjugate.

Set $X:=U{\uparrow^{T^*}}$. Then $V=X{\uparrow^G}$, and $X$ is an irreducible submodule of $V{\downarrow_{T^*}}$. Now $U$ is a submodule of $X{\downarrow_T}$. So by uniqueness of $U$, $X$ is the unique irreducible submodule of $V{\downarrow_{T^*}}$ lying over $W$. Likewise $X^*$ is the unique irreducible submodule of $V{\downarrow_{T^*}}$ lying over $W^*$. But $X$ lies over $W^*$, as $W$ and $W^*$ are $T^*$-conjugate. So $X\cong X^*$.

Let $\tau\in T^*\backslash T$. Then $X{\downarrow_T}=U\oplus\tau U$. But $U$ and $U^*$ are non-isomorphic irreducible submodules of $X{\downarrow_T}$. So $X{\downarrow_T}=U\oplus U^*$, whence $\tau U\cong U^*$.

Let $B$ be a $G$-invariant non-degenerate alternating bilinear form on $V$, and let $X^\perp$ be the orthogonal complement of $X$ in $V{\downarrow_{T^*}}$ with respect to $B$. Then $X\cong X^*\cong V/X^\perp$. So $X\cap X^\perp=0$, by uniqueness of $X$. This shows that the restriction $B_X$ of $B$ to $X$ is a ($T^*$-invariant) non-degenerate alternating bilinear form on $X$. Since $U$ is irreducible but not self-dual, it is totally isotropic with respect to $B_X$, and likewise, so is $\tau U$. We define $Q:X\rightarrow F$ via
$$
Q(u_1+\tau u_2)=B(u_1,\tau u_2),\quad\mbox{for all $u_1,u_2\in U$.}
$$
Then $Q$ is a quadratic form which polarizes to $B_X$. As $Q$ vanishes on the subspaces $U$ and $\tau U$, it is an example of a hyperbolic form.

We now check that $Q$ is $T^*$-invariant. It is certainly $T$-invariant, as $T$ fixes $U$ and $\tau U$, and preserves $B$. Suppose that $\tau'\in T^*\backslash T$. Then $\tau' u_1\in \tau U$ and $\tau'\tau u_2\in U$. So
$$
Q(\tau'(u_1+\tau u_2))=B(\tau'\tau u_2,\tau' u_1)=B(\tau u_2,u_1)=B(u_1,\tau u_2)=Q(u_1+\tau u_2),
$$
since $\tau'$ also leaves $B$ invariant. So $Q$ is indeed $T^*$-invariant.

Finally, the induced form $Q{\uparrow^G}$ is a $G$-invariant quadratic form on $V$ which polarizes to the $G$-invariant non-degenerate alternating bilinear form $B_X{\uparrow^G}$. So $V$ is of quadratic type, as required.
\end{proof}

\begin{proof}[Proof of Theorem \ref{T:non_quadratic}]
Let $V$ be a self-dual irreducible $FG$-module which is not of quadratic type, such that $N$ does not act trivially on $V$. Write $V{\downarrow_N} = e(W_1\oplus\dots\oplus W_t)$, where $e>0$ and $W_1,\dots,W_t$ is a $G$-orbit of irreducible $FN$-modules, each of which is irreducible and non-trivial.

It follows from Proposition \ref{P:hyperbolic_quadratic} that each $W_i$ is self-dual, and then from Proposition \ref{P:even_quadratic_type} that $e=1$. So $V$ is the canonical self-dual irreducible $FG$-module over each $W_i$. Then Proposition \ref{P:canonical_quadratic_type} implies that each $W_i$ has non-quadratic type.
\end{proof}

We now show how the techniques we have developed above can be employed to obtain a criterion for all self-dual irreducible $FG$-modules to be of quadratic type.

\begin{Corollary}\label{C:necessary_and_sufficient_condition}
Let $N$ be a normal subgroup of\/ $G$. Then all non-trivial self-dual irreducible $FG$-modules are of quadratic type if and only if the same is true of both $FN$ and $FG/N$.
\end{Corollary}

\begin{proof}
Suppose first that all non-trivial self-dual irreducible $FG$-modules are of quadratic type. Then the same is obviously true for $G/N$ and we must show that it is true for $N$. Let $W$ be a non-trivial self-dual  irreducible $FN$-module and let $V$ be the canonical self-dual irreducible $FG$-module over $W$. Then $V$ is  irreducible and of quadratic type, from the hypothesis. So $W$ is of quadratic type, according to Proposition \ref{P:canonical_quadratic_type}.

Conversely, suppose that all non-trivial self-dual irreducible modules of both $N$ and $G/N$ are of quadratic type. Then for the sake of contradiction, suppose that $V$ is a self-dual irreducible $FG$-module which has non-quadratic type. Theorem \ref{T:non_quadratic} implies that $N$ acts trivially on $V$. So $V$ is a self-dual irreducible $F(G/N)$-module of non-quadratic type, contrary to hypothesis.
\end{proof}

\section{Irreducible self-dual modules of non-abelian finite simple groups}

\noindent The proof of Corollary \ref{C:necessary_and_sufficient_condition} shows that in an inductive approach to deciding whether a self-dual irreducible $FG$-module is of quadratic type, the main difficulty lies in solving the problem for non-abelian simple groups, and as far as we know, this is an unsolved and difficult question. See \cite[Remark 3.4(a)]{W}.

At a simpler, but by no means straightforward, level, we can ask if every non-abelian simple group has a non-trivial irreducible module of quadratic type. The answer is no, for according to \cite{HM}, the Mathieu simple group $M_{22}$ has no such modules. We were unable to find an explicit reference to the calculations needed to verify this in the literature. So we outline a proof here, which only assumes some knowledge of the irreducible Brauer characters of certain groups. 

We use the notation and decomposition matrices from \cite{MA} and character tables from \cite{GAP4}. So $M_{22}$ has exactly two non-trivial self-dual irreducible Brauer characters $\phi_4$ and $\phi_7$, with degrees $34$ and $98$, respectively.

Now $M_{22}$ has two conjugacy classes of maximal subgroups isomorphic to the alternating group $A_7$. The restriction of $\phi_4$ to any $A_7$ is the sum of an irreducible character $\psi$ of degree $20$ plus another of degree $14$. In turn, the restriction of $\psi$ to $A_6$ is the sum of an irreducible character $\mu$ of degree $4$ plus two irreducible characters of degree $8$. Examining the values of $\mu$, we see that it is the Brauer character of a representation defined over ${\mathbb F}_2$. So $\mu$ cannot be of quadratic type. For the order of $A_6$ is greater than the order of each of the two orthogonal groups $O_+(4,2)\cong S_3\wr C_2$ and $O_-(4,2)\cong S_5$. It now follows from \cite[Lemma 1.2]{GW95} that $\phi_4$ is not of quadratic type.

Next we observe that $\phi_2\phi_3=2\phi_1+\phi_7$. Here $\phi_2$ and $\phi_3=\ov\phi_4$ are of degree $10$, and $\phi_1$ is the trivial Brauer character. Now $\phi_2\phi_3$ is the Brauer character of the ring $E$ of $F$-endomorphisms of a module affording $\phi_2$. Let $\op{Tr}:E\rightarrow F$ be the trace map and let $W=\{A\in E\mid\op{Tr}(A)=0_F\}$. Then $W$ is a submodule of $E$ and $E/W$ is the trivial module. So $W$ has Brauer character $\phi_1+\phi_7$. Clearly the identity map $I\in E$ spans the unique trivial submodule of $W$. Now for each $A\in W$, set $Q(A)$ as the coefficient of $x^{n-2}$ in the $F$-characteristic polynomial of $A$. Then $Q$ is a $G$-invariant quadratic form, with polarization $B(A,B)=\op{Tr}(AB)$, for all $A,B\in W$. In particular $I$ spans $\op{Rad}(Q)$. As $I$ has characteristic polynomial $(x-1)^{10}=x^{10}+x^8+x^2+1$, we see that $Q(I)=1_F$. So the singular radical $\op{Rad}_0(Q)$ is $0$. Now \cite[Theorem 1.3]{GW95} implies that $\phi_7$ is not of quadratic type.

On the other hand, each of the remaining 25 sporadic finite simple groups does have irreducible $FG$-modules of quadratic type. To show this, we need \cite[Corollary IV.11.9]{F}. We give an elementary proof here, for the convenience of the reader:

\begin{Lemma}
Let $G$ be a finite group and let $F$ be a perfect field of characteristic $2$. Then each self-dual but non-quadratic irreducible $FG$-module is in the principal $2$-block of\/ $G$.
\begin{proof} Let $M$ be a self-dual irreducible $FG$-module which is not in the principal $2$-block of $G$. We aim to show that $M$ has quadratic type. For this, we exploit the fact that there is no module $W$ such that $\op{soc}(W)$ is the trivial module and $\frac{W}{\op{soc}(W)}\cong M$.

Let $B$ be a $G$-invariant non-degenerate symplectic bilinear form on $M$, and let $Q$ be a quadratic form on $M$ which polarizes to $B$. For all $g\in G$, define $Q^g(m):=Q(gm)$. Then $Q^g$ is a quadratic form which polarizes to $B$, as
$$
Q^g(m_1+m_2)=Q(gm_1)+B(gm_1,gm_2)+Q(gm_2)=Q^g(m_1)+B(m_1,m_2)+Q^g(m_2).
$$

Consider the quadratic form $Q+Q^g$. This is additive and satisfies $(Q+Q^g)(\lambda m)=\lambda^2(Q+Q^g)(m)$, for all $\lambda\in F$ and $m\in M$. As $F$ is perfect, there exists $\phi_g\in M^*$ such that $Q(gm)=Q(m)+\phi_g(m)^2$, for all $m\in M$. Now for $g,h\in G$, and $m\in M$ we have
$$
\begin{aligned}
Q(m)+\phi_{gh}(m)^2&=Q(ghm)\\
&=Q(hm)+\phi_g(hm)^2\\
&=Q(m)+\phi_h(m)^2+h^{-1}\phi_{g}(m)^2\\
\end{aligned}
$$
So $\phi:G\rightarrow M^*$ satisfies the cocycle condition $\phi_{gh}=\phi_h+h^{-1}\phi_g$, for all $g,h\in G$.

Now take $W$ to be the Cartesian product $M\times F$, endowed with the obvious $F$-vector space structure. Define an $FG$-module structure on $W$ via
$$
g(m,\lambda)=(gm,\phi_g(m)+\lambda),\quad\mbox{for all $m\in M, \lambda\in F$ and $g\in G$.}
$$
This is an action because for all $g,h\in G$, we have
$$
\begin{aligned}
(gh)(m,\lambda)&=(ghm,\phi_{gh}(m)+\lambda)\\
&=(ghm,\phi_g(hm)+\phi_h(m)+\lambda)\\
&=g(hm,\phi_h(m)+\lambda)\\
&=g(h(m,\lambda)).
\end{aligned}
$$
Clearly $W$ has a submodule $0\times F$ isomorphic to the trivial $FG$-module $F$, and $W$ modulo this submodule is isomorphic to $M$. Our assumption on $M$ forces $W\cong M\oplus F$ as $FG$-modules. So there is $\psi\in M^*$ such that $m\mapsto(m,\psi(m))$, for $m\in M$, is an injective $FG$-module map $M\rightarrow W$.

Now on the one hand $g(m,\psi(m))=(gm,\psi(gm))$. On the other hand $g(m,\psi(m))=(gm,\phi_g(m)+\psi(m))$. Comparing these expressions, we see that $\phi_g(m)+\psi(gm)=\psi(m)$, for all $g\in G$. Finally, define the quadratic form $\hat Q$ on $M$ via
$$
\hat Q(m)=Q(m)+\psi(m)^2,\quad\mbox{for all $m\in M$.}
$$
Then it is clear that $\hat Q$ polarizes to $B$. Furthermore, for all $g\in G$ we have
$$
\hat Q(gm)=Q(gm)+\psi(gm)^2=Q(m)+\phi_g(m)^2+\psi(gm)^2=Q(m)+\psi(m)^2=\hat Q(m).
$$
So $\hat Q$ is $G$-invariant.
\end{proof}
\end{Lemma}

Using \cite{GAP4} and \cite{MA}, the only sporadic finite simple groups which do not have a real non-principal $2$-block are $M_{11},M_{22},M_{23}$ and $M_{24}$. Now $M_{11}$ has an orthogonal irreducible $K$-character $\chi_2$, of degree 10, whose restriction to $2$-regular elements is a self-dual irreducible Brauer character $\phi_2$. So $\phi_2$ has quadratic type. Similarly $M_{24}$ has an orthogonal irreducible $K$-character $\chi_7$, of degree 252, whose restriction to $2$-regular elements contains the self-dual irreducible Brauer character $\phi_6$ with multiplicity 1, but does not contain the trivial Brauer character. So $\phi_6$ has quadratic type. Finally, $\phi_6$ restricts to an irreducible Brauer character of a maximal subgroup $M_{23}$. So $M_{23}$ also has a quadratic type irreducible Brauer character.

All other simple group whose modular representations are tabulated in the modular \cite{A} have quadratic type irreducible Brauer characters, and we suspect that $M_{22}$ may be unique among all non-abelian finite simple groups in not having such a character. We note that the automorphism group of $M_{22}$ does have irreducible modules of quadratic type, since Proposition \ref{P:hyperbolic_quadratic} applies to certain irreducible modules of the automorphism group that are induced from irreducible modules of $M_{22}$ that are not self-dual.

\section{Real weakly regular $2$-blocks}\label{S:blockCovering}

\noindent We continue to assume that $G$ is a finite group and $N$ is a normal subgroup of $G$. The results in this section include real refinements of \cite[Theorem 4.4, Corollary 4.5]{M}. If $C$ is a conjugacy class of $G$, then $C^+$ is the sum of its elements in $RG$. Also $C^o$ is the class consisting of the inverses of the elements of $C$. Each $z\in\op{Z}(FG)$ can be written $z=\sum\beta(z,C)C^+$, where $C$ ranges over the conjugacy classes of $G$ and $\beta(z,C)\in F$. 

We use standard notation and results on blocks. In particular, corresponding to each $2$-block $B$ of $G$, there is a  primitive idempotent $e_B$ of the centre $\op{Z}(FG)$ of $FG$, an $F$-algebra homomorphism $\omega_B:\op{Z}(FG)\rightarrow F$, called the central character of $B$, and a $2$-subgroup $D$ of $G$ called a defect group of $B$. Then $D$ is only determined up to $G$-conjugacy, and $|D|=2^d$, where $d\geq0$ is called the defect of $B$. We use $\op{Irr}(B)$ and $\op{IBr}(B)$ to denote the irreducible $K$-characters and irreducible Brauer characters in $B$, respectively.

Let $\chi\in\op{Irr}(B)$, let $\psi$ be an irreducible constituent of $\chi{\downarrow_N}$ and let $b$ be the $2$-block of $N$ containing $\psi$. Then $B$ is said to cover $b$, and the $2$-blocks of $N$ covered by $B$ form a single $G$-orbit. Set $e_b^G$ as the sum of the distinct $G$-conjugates of $e_b$. Then $e_b^G$ is an idempotent in $\op{Z}(FG)$ which is the sum of the block idempotents of all blocks of $G$ which cover $b$.

Recall that $B$ is said to be weakly regular (with respect to $N$) if it has maximal defect among the set of blocks of $G$ which cover $b$. This happens if and only if $B$ has a defect group $D$ such that $DN/N$ is a Sylow $2$-subgroup of the stabilizer of $b$ in $G$.

Let $\chi$ be a $K$-character or Brauer character belonging to $B$. Then $\chi(1)_2\geq|G:D|_2$. If equality occurs, we say that $\chi$ has height $0$. Recall that if $\chi$ is irreducible, its central character is defined by $\omega_\chi(C^+):=\chi(C^+)/\chi(1)$, for all conjugacy classes $C$ of $G$. It is classical that $\omega_\chi(C^+)\in R$. Indeed its image $\omega_\chi(C^+)^*$ in $F$ is independent of $\chi\in\op{Irr}(B)$, as it equals $\omega_B(C^+)$. Suppose now that $\theta\in\op{IBr}(B)$ has height $0$. We claim that for all $2$-regular conjugacy classes $C$ of $G$
\begin{equation}\label{E:central_theta}
\frac{\theta(C^+)}{\theta(1)}\in R
\qquad
\mbox{and}
\qquad
\left(\frac{\theta(C^+)}{\theta(1)}\right)^*=\omega_B(C^+).
\end{equation}
For, it is known that there are integers $n_\chi$ such that $\theta\equiv\sum_{\chi\in\op{Irr}(B)}n_\chi\chi$ on the $2$-regular elements of $G$. As $\chi(C^+)/\chi(1)$ and $\chi(1)/\theta(1)$ belong to $R$, we get
$$
\frac{\theta(C^+)}{\theta(1)}=\sum_{\chi\in\op{Irr}(B)}\left(\frac{\chi(C^+)}{\chi(1)}\right)\left(\frac{n_\chi\chi(1)}{\theta(1)}\right)
\qquad\mbox{belongs to $R$.}
$$
Moreover $\left(\frac{\theta(C^+)}{\theta(1)}\right)^*=\omega_B(C^+)\left(\frac{\sum_{\chi\in\op{Irr}(B)}n_\chi\chi(1)}{\theta(1)}\right)^*=\omega_B(C^+)$.

Our first result includes a proof of part (i) of Theorem \ref{T:block_covering}:

\begin{Lemma}\label{L:block_covering_part1}
Let $b$ be a $2$-block of\/ $N$. Then $G$ has an odd number of weakly regular $2$-blocks covering $b$. So $G$ has a real weakly regular $2$-block covering $b$ if and only if\/ $b$ is $G$-conjugate to $b^o$. 

Let $B$ be a weakly regular $2$-block covering $b$. Then $\beta(e_B,C)\omega_B(C^+)=\beta(e_b^G,C)\omega_b(C^+)$, for all conjugacy class $C$ of\/ $G$ contained in $N$.
\begin{proof}
The first statement is proved in Lemma 5.1 of \cite{GM}, so we merely summarize the argument here. There is a defect preserving bijection between the blocks of $G$ covering $b$ and the blocks of the $G$-stabilizer of $b$ covering $b$. So we may assume that $b$ is $G$-invariant.

Let $B$ be as in the statement. In particular $e_B=e_Be_b$. So $1_F=\omega_B(e_B)=\omega_B(e_b)$. Thus there is a conjugacy class $L$ of $G$ contained in $N$ such that $\beta(e_b,L)\omega_B(L^+)\ne0_F$. Now $L$ is $2$-regular, as it is in the support of the block idempotent $e_b$. As $e_b$ is a sum of block idempotents of blocks of $G$ with a defect group contained in $D$, $L$ has a defect group contained in $D$. But $\omega_B(L^+)\ne0_F$. So $L$ has a defect group containing the defect group $D$ of $B$. We deduce that $D$ is a defect group of $L$.

Corollary 3.2 of \cite{GM} implies that $\beta(e_B,L)=\omega_B(L^{o+})$. But $\omega_B(L^{o+})=\omega_{B'}(L^{o+})$, for each block $B'$ of $G$ which covers $b$, as $L\subseteq N$. So, again by Corollary 3.2 of \cite{GM} $\beta(e_B,L)=\beta(e_{B'},L)$, if $B'$ is in addition weakly regular. On the other hand $\beta(e_{B'},L)=0_F$, if $B'$ is not weakly regular. As $e_b$ is the sum of the block idempotents of all blocks of $G$ covering $b$, we see that $\beta(e_b,L)=\beta(e_B,L)\rho$, where $\rho$ is the number of weakly regular $2$-blocks of $G$ covering $b$. It follows from this that $\rho$ is odd.

Suppose that there is a real weakly regular $2$-block $B$ of\/ $G$ which covers $b$. Then $B=B^o$ also covers $b^o$. So $b$ is $G$-conjugate to $b^o$. Conversely, suppose that $b$ is $G$-conjugate to $b^o$. Then taking contragredients of blocks is an involution on the set of weakly regular $2$-block of\/ $G$ covering $b$. As this set has odd size $\rho$, we deduce that there is a real weakly regular $2$-block of\/ $G$ which covers $b$.

For the last statement, let $C$ be a conjugacy class of\/ $G$ which is contained in $N$ for which $\beta(e_B,C)\omega_B(C^+)\ne0_F$ or $\beta(e_b,C)\omega_b(C^+)\ne0_F$. As $\omega_B(C^+)=\omega_b(C^+)$, the argument above implies that $D$ is a defect group of $C$. But then $\beta(e_b,C)=\beta(e_B,C)\rho=\beta(e_B,C)$, as $\op{char}(F)=2$. We conclude that $\beta(e_B,C)\omega_B(C^+)=\beta(e_b,C)\omega_b(C^+)$.
\end{proof}
\end{Lemma}

We need one more result before proving part (ii) of Theorem \ref{T:block_covering}:

\begin{Lemma}\label{L:invariant_real_height0}
Let $b$ be a real $G$-invariant $2$-block of\/ $N$. Then $G$ has a self-dual Brauer character $\phi$ such that $\phi$ vanishes off\/ $N$ and $\phi{\downarrow_N}=e(\theta_1+\dots+\theta_t)$ where both $e$ and $t$ are odd and $\theta_1,\dots,\theta_t$ are distinct self-dual height $0$ irreducible Brauer characters in $b$.
\begin{proof}
Note that we are not claiming that $\phi$ is irreducible.

Consider the $G$-set $X:=\{\theta\in\op{IBr}(b)\mid\mbox{$\theta$ has height zero and }\Phi_\theta(1)_2=|N|_2\}$. Then $|X|$ is odd, according to Lemma \ref{L:odd_height0}. Also duality is an involution on $X$. So there is a $G$-orbit $\theta_1,\dots,\theta_t$ in $X$, with $t$ odd and all $\theta_i$ self-dual and of height $0$.

Let $T$ be the inertial group of $\theta_1$ in $G$. Then $T$ contains a Sylow $2$-subgroup $S$ of $G$. As $SN/N$ is a $2$-group, $\theta_1$ has a unique extension $\hat\theta_1$ to an irreducible Brauer character of $SN$. Notice that $\hat\theta_1$ vanishes off $N$, as $N$ contains all $2$-regular elements of $SN$.  

Set $\phi:={\hat\theta_1}{\uparrow^G}$. Then $\phi$ is self-dual and $\phi{\downarrow_N}=\frac{1}{[SN:N]}({\theta_1}{\uparrow^G}){\downarrow_N}=e(\theta_1+\dots+\theta_t)$, where $e=[T:SN]$ is odd. Finally $\phi$ vanishes off $N$ as $\hat\theta_1$ vanishes off $N$.
\end{proof}
\end{Lemma}

We now prove the uniqueness part (ii) of Theorem \ref{T:block_covering}:

\begin{Lemma}\label{L:block_covering_part2}
Let $b$ be a real $2$-block of\/ $N$. Then $G$ has a unique real $2$-block which covers $b$ and which is weakly regular with respect to $N$.
\end{Lemma}

\begin{proof}
We may assume that $b$ is $G$-invariant, and we let $B$ be any real weakly regular $2$-block of $G$ covering $b$. Let $\phi$ be the Brauer character of $G$ defined in Lemma \ref{L:invariant_real_height0}. So $\phi{\downarrow_N}=e(\theta_1\dots+\theta_t)$ where $et$ is odd and $\theta_1,\dots,\theta_t$ are distinct self-dual height $0$ irreducible Brauer characters in $b$. Write $\phi=\sum_{\mu\in\op{IBr}(G)}m_\mu\mu$, where $m_\mu$ are non-negative integers. Then $\phi_B:=\sum_{\mu\in\op{IBr}(B)}m_\mu\mu$ is the $B$-part of $\phi$.

Let $C$ be a $2$-regular conjugacy class of $G$ which is contained in $N$. Then $\theta_i(C^+)=\theta_1(C^+)$, for $i=1,\dots,t$, as $\theta_i$ is $G$-conjugate to $\theta_1$. So
$$
\left(\frac{\phi(C^+)}{\theta_1(1)}\right)^*=\left(\frac{et\theta_1(C^+)}{\theta_1(1)}\right)^*=\omega_b(C^+)=\omega_B(C^+),
$$
where we have used \eqref{E:central_theta}.

Next let $\hat{e_B}$ be the unique idempotent in $\op{Z}(RG)$ with $\hat{e_B}^*=e_B$. Then for all $\mu\in\op{IBr}(G)$ we have $\mu(\hat{e_B})=\mu(1)$ or $0_R$, as $\mu$ does or does not belong to $B$, respectively. So
$$ 
\left(\frac{\phi_B(1)}{\theta_1(1)}\right)^*=\left(\frac{\phi(\hat{e_B})}{\theta_1(1)}\right)^*=\sum\beta(e_B,C^+)\omega_B(C^+)=\sum\beta(e_b,C^+)\omega_b(C^+)=\omega_b(e_b)=1_F.
$$
Here in both sums, $C$ ranges over the conjugacy classes of $G$ which are contained in $N$, as $\phi$ vanishes off $N$. Also the middle equality arises from the last assertion in Lemma \ref{L:block_covering_part1}.

Now for each $\mu\in\op{IBr}(B)$ with $m_\mu\ne0$, we have $\mu{\downarrow_N}=e_\mu(\theta_1\dots+\theta_t)$, for some integer $e_\mu>0$. Then by the previous displayed equation
$$
\frac{\phi_B(1)}{\theta_1(1)}=t\,\sum_{\mu\in\op{IBr}(B)}m_\mu e_\mu\quad\mbox{is an odd integer.}
$$
As $m_\mu e_\mu=m_{\ov{\mu}}e_{\ov{\mu}}$, it follows that there is a self-dual $\mu\in\op{IBr}(B)$ such that $m_\mu e_\mu$ is odd. Then $\mu$ is the canonical irreducible Brauer character of $G$ lying over $\theta_1$ given by Theorem \ref{T:real_extends}. As $\theta_1$ determines $\mu$, which in turn determines $B$, we conclude that $B$ is the only real weakly regular $2$-block of $G$ which covers $b$, as we wished to show.
\end{proof}


\end{document}